\documentclass[12pt,a4paper]{amsart}
\usepackage[italian, UKenglish]{babel}
\usepackage[latin9]{inputenx} 
\usepackage{fontenc}
\usepackage[mathscr]{eucal}
\usepackage{indentfirst}
\usepackage{newlfont}
\usepackage{verbatim}
\usepackage{lmodern,textcomp}
\usepackage{fp}
\usepackage{booktabs}
\usepackage{ragged2e}
\usepackage{longtable}
\usepackage{epic}
\usepackage[margin=2.9cm]{geometry}
\usepackage{epstopdf} 
\usepackage{multicol}
\usepackage{xparse}
\usepackage{multicol}
\usepackage{xcolor}
\usepackage{colortbl}
\usepackage[most]{tcolorbox}
\usepackage{amsfonts}
\usepackage{amsthm}
\numberwithin{equation}{section}
\usepackage{amsmath}
\usepackage{amscd}
\usepackage{amssymb}
\usepackage{mathtools}
\usepackage{leftidx}
\usepackage{mathabx}
\usepackage{accents}
\usepackage[shortlabels]{enumitem}
\usepackage[makeroom]{cancel}
\usepackage{multirow,bigdelim}
\usepackage{array}
\usepackage{hhline}
\usepackage{mathptmx}
\usepackage{calc}
\usepackage{hyperref} 
\makeatletter
\def\@tocline#1#2#3#4#5#6#7{\relax
  \ifnum #1>\c@tocdepth 
  \else
    \par \addpenalty\@secpenalty\addvspace{#2}%
    \begingroup \hyphenpenalty\@M
    \@ifempty{#4}{%
      \@tempdima\csname r@tocindent\number#1\endcsname\relax
    }{%
      \@tempdima#4\relax
    }%
    \parindent\z@ \leftskip#3\relax \advance\leftskip\@tempdima\relax
    \rightskip\@pnumwidth plus4em \parfillskip-\@pnumwidth
    #5\leavevmode\hskip-\@tempdima
      \ifcase #1
       \or\or \hskip 1em \or \hskip 2em \else \hskip 3em \fi%
      #6\nobreak\relax
    \dotfill\hbox to\@pnumwidth{\@tocpagenum{#7}}\par
    \nobreak
    \endgroup
  \fi}
\makeatother
\NewDocumentCommand{\diamondinclusion}{m >{\SplitArgument{1}{\\}}m m}
 {%
  \dodiamondinclusion{#1}#2{#3}%
 }
\NewDocumentCommand{\dodiamondinclusion}{mmmm}
 {%
  \begingroup
  \setlength{\arraycolsep}{0pt}%
  \renewcommand{\arraystretch}{-2}%
  \begin{matrix}
  && #2 \\
  & \rsubsetneq{45} && \rsubsetneq{-45} \\
  #1 &&&& #4 \\
  & \rsubsetneq{-45} && \rsubsetneq{45} \\
  && #3
  \end{matrix}%
  \endgroup
 }
\NewDocumentCommand{\rsubsetneq}{m}
 {%
  \rotatebox[origin=c]{#1}{$\subsetneq$}%
 }

\newcommand{\mycomment}[1]{%
}

\newcommand*{\norm}[1]{\left\lVert#1\right\rVert}
\newcommand*{\vertiii}[1]{{\left\vert\kern-0.25ex\left\vert\kern-0.25ex\left\vert #1 \right\vert\kern-0.25ex\right\vert\kern-0.25ex\right\vert}}

\makeatletter
\newcommand\mathcircled[1]{%
  \mathpalette\@mathcircled{#1}%
}
\newcommand\@mathcircled[2]{%
  \tikz[baseline=(math.base)] \node[draw,circle,inner sep=1pt] (math) {$\m@th#1#2$};%
}
\makeatother
\newtheorem{innercustomgeneric}{\customgenericname}
\providecommand{\customgenericname}{}
\newcommand{\newcustomtheorem}[2]{%
  \newenvironment{#1}[1]
  {%
   \renewcommand\customgenericname{#2}%
   \renewcommand\theinnercustomgeneric{##1}%
   \innercustomgeneric
  }
  {\endinnercustomgeneric}
}
\newcustomtheorem{customprop}{Proposition}
\newcustomtheorem{customcor}{Corollary}
\newcustomtheorem{customdefin}{Definition}
 
\allowdisplaybreaks 
\DeclareMathAlphabet{\mathcal}{OMS}{cmsy}{m}{n}

\DeclareMathOperator{\spn}{span}

\theoremstyle{plain}
\newtheorem{theor}{Theorem}[section] 
\newtheorem{lem}[theor]{Lemma} 
\newtheorem{cor}[theor]{Corollary}
\newtheorem{prop}[theor]{Proposition}
\newtheorem*{cor*}{Corollary}
\newtheorem*{prop*}{Proposition}
\theoremstyle{definition}
\newtheorem{defin}[theor]{Definition}
\newtheorem*{defin*}{Definition}
\newtheorem{rem}[theor]{Remark}
\newtheorem{ex}[theor]{Example}

\newtheorem{no}[theor]{Notation}

\begin{document}
\title{Notion of $\mathbb{H}$-orientability for surfaces in the Heisenberg group $\mathbb{H}^n$}
\author[G. Canarecci]{Giovanni Canarecci}
\address{University of Helsinki \\ Department of Mathematics and Statistics \\
Helsinki, Finland} 
\email{giovanni.canarecci@helsinki.fi}
\begin{abstract} 
This paper aims to define and study a notion of orientability in the Heisenberg sense ($\mathbb{H}$-\emph{orientability}) for the Heisenberg group $\mathbb{H}^n$. In particular, we define such notion for $\mathbb{H}$-regular $1$-codimensional surfaces. 
Analysing the behaviour of a M{\"o}bius Strip in $\mathbb{H}^1$, 
we find a $1$-codimensional $\mathbb{H}$-regular, but not Euclidean-orientable, subsurface.  
Lastly we show that, for regular enough surfaces, $\mathbb{H}$-orientability implies Euclidean-orientability. As a consequence, we conclude that non-$\mathbb{H}$-orientable $\mathbb{H}$-regular surfaces exist in $\mathbb{H}^1$.
\end{abstract}
\maketitle
\tableofcontents


\section*{Introduction} 
The aim of this paper is to define and study a notion of orientability in the Heisenberg sense ($\mathbb{H}$-\emph{orientability}) for the Heisenberg group $\mathbb{H}^n$. 
There exist many references for an introduction about the Heisenberg group; here we use, for example, parts of \cite{GCmaster}, \cite{CDPT}, \cite{FSSC} and \cite{TRIP}.  
The Heisenberg group $\mathbb{H}^n$, $n \geq 1$, is the $(2n+1)$-dimensional manifold $ \mathbb{R}^{2n+1}$ with a non-Abelian group product and  the Carnot--Carath{\'e}odory distance.   
Such a group has two important automorphisms playing a role in its geometry: left translations and anisotropic dilations. 
Additionally, the Heisenberg group is a Carnot group of step $2$ with Lie algebra $\mathfrak{h} = \mathfrak{h}_1 \oplus \mathfrak{h}_2$. 
The \textit{horizontal} layer $\mathfrak{h}_1$ has a standard orthonormal basis of left invariant vector fields,  
$ X_j=\partial_{x_j} -\frac{1}{2} y_j \partial_t$ and $Y_j=\partial_{y_j} +\frac{1}{2} x_j \partial_t$ for $ j=1,\dots,n$,     
which hold the core property that $[X_j, Y_j] = \partial_t=:T $ for each $j$. 
$T$ alone spans the second layer $\mathfrak{h}_2$ and is called the \textit{vertical} direction. By definition, the horizontal subbundle changes inclination at every point, allowing movement from any point to any other point following only horizontal paths; this allows to define the Carnot--Carath{\'e}odory distance $d_{cc}$, measured along curves whose tangent vector fields are horizontal. An equivalent standard distance in $\mathbb{H}^n$ is the Kor{\'a}nyi distance. The topological dimension of the Heisenberg group is $2n+1$, while its Hausdorff dimension with respect to the Carnot--Carath{\'e}odory and Kor{\'a}nyi distances is $2n+2$; such difference hints to the existence of a natural cohomology called Rumin cohomology 
(see Rumin \cite{RUMIN}), whose behaviour is significantly different from the standard de Rham one (see also \cite{GClicentiate}).  Another consequence of the dimensional difference is the existence of surfaces regular in the Heisenberg sense but fractal in the Euclidean sense (see Kirchheim and Serra Cassano \cite{KSC}).\\
In the second part of the paper we discuss the notion of $\mathbb{H}$-regularity for low dimensional and low codimensional surfaces in the Heisenberg group (see Franchi, Serapioni and Serra Cassano \cite{FSSC2001} and \cite{FSSC}), which, in the codimensional case, requires the surface to be locally the level set of a function with non-vanishing horizontal gradient. The points where such gradient is null are called characteristic (see, for instance, Balogh \cite{BAL} and Magnani \cite{MAG}) and cannot be part of a $\mathbb{H}$-regular surface. After that, we analyse the behaviour of a specific M{\"o}bius Strip $\mathcal{M} \subseteq \mathbb{H}^1$ and find a subset $\widetilde{\mathcal{M}}\subseteq \mathcal{M}$ that is a $1$-codimensional $C^1$-Euclidean surface with no characteristic points ($ C( \widetilde{\mathcal{M}} ) = \varnothing  $)  and non-orientable in the Euclidean sense:
\begin{prop*}[\ref{Mobius}]
The M\"obius strip $\mathcal{M}$ contains at most one characteristic point $\tilde p$ and so there exists a $1$-codimensional  $C^1$-Euclidean surface $\widetilde{\mathcal{M}}  \subseteq \mathcal{M}$ such that $\tilde p  \notin \widetilde{\mathcal{M}}$, $\widetilde{\mathcal{M}}$ is still non-orientable in the Euclidean sense and  $C( \widetilde{\mathcal{M}} ) = \varnothing$.
\end{prop*}
In particular $\widetilde{\mathcal{M}}$ is a $1$-codimensional $\mathbb{H}$-regular surface which is non-orientable in the Euclidean sense (this is done in Subsection \ref{subsec:Mobius}). The idea is to consider a M\"obius strip surface $\mathcal{M}$ which is $1$-codimensional in $\mathbb{H}^1$ and $C^1$-Euclidean; given this surface, we check the existence of its characteristic points, meaning those points where $\mathcal{M}$ fails to respect the $\mathbb{H}$-regularity condition. \\
For $1$-codimensional $C^1$-Euclidean surfaces with no characteristic points we give a new definition of orientability ($\mathbb{H}$-orientability) as follows:
\begin{defin*}[\ref{def:Horientable}]
Consider a $1$-codimensional $C^1$-Euclidean surface $S \subseteq \mathbb{H}^n$ with $C(S) = \varnothing$. We say that $S$ is $\mathbb{H}$\emph{-orientable} (or \emph{orientable in the Heisenberg sense})  if there exists a continuous global $1$-vector field 
$$
n_{\mathbb{H}}=\sum_{i=1}^{n} \left ( n_{\mathbb{H},i} X_i + n_{\mathbb{H},n+i} Y_i \right ) \neq 0,
$$
defined on $S$ so that $n_{\mathbb{H}} \perp_H S$.
\end{defin*}
As expected, such definition is invariant under left translations $\tau_q p =q*p$ and the anisotropic dilations $\delta_r(x,y,t)=(rx,ry,r^2t)$ for $\mathbb{H}$-regular $1$-codimensional surfaces.\\
Lastly we show that, for regular enough surfaces, $\mathbb{H}$-orientability implies orientability in the Euclidean sense, while the other direction requires more hypotheses:
\begin{prop*}[\ref{finalmente4}]
Consider a $1$-codimensional $C^1$-Euclidean surface $S$ in $\mathbb{H}^{n}$ with $C(S) = \varnothing$. If $S$ is  $C^2_\mathbb{H}$-regular, then: $S$  is  $\mathbb{H}$-orientable implies $S$  is Euclidean-orientable.
\end{prop*}
 As a consequence, we conclude that non-$\mathbb{H}$-orientable $\mathbb{H}$-regular surfaces exist in $\mathbb{H}^1$:
\begin{cor*}[\ref{cor:existence}]
There exist $\mathbb{H}$-regular surfaces which are not $\mathbb{H}$-orientable in $\mathbb{H}^1$.
\end{cor*}
One of the reasons behind this study is the important role that orientability plays in the theory of currents.  
Standard theory of currents requires orientability for certain surfaces although, in Riemannian geometry, there exists a notion of currents for surfaces that are not necessarily orientable (see, for instance, Morgan \cite{MORGAN2}).   
In the Heisenberg group, $\mathbb{H}$-regular $\mathbb{H}$-orientable surfaces can still be associated to currents, but it was not known whether it was meaningful to study a second notion for not necessarily orientable surfaces.  
We have shown that it is indeed a meaningful task.\\

\noindent
{\bf Acknowledgments.}  
I would like to thank my adviser, university lecturer Ilkka Holopainen, for the work done together and the time dedicated to me. I also want to thank professor Bruno Franchi, for the time I spent in Bologna, and professors Raul Serapioni and Pierre Pansu, for the stimulating discussions.


\section{Preliminaries}

In this section we introduce the Heisenberg group $\mathbb{H}^n$, its structure as a Carnot group and the standard bases of vector fields and differential forms in $\mathbb{H}^n$. Then we mention the standard \emph{Kor{\'a}nyi} and \emph{Carnot--Carath{\'e}odory} distances.\\
There exist many good references for an introduction on the Heisenberg group; we follow mainly sections 2.1 and 2.2 in \cite{FSSC} and section 2.1.3 and 2.2 in \cite{CDPT}.

\subsection{The Heisenberg Group $\mathbb{H}^n$}\label{defH}

\begin{defin}\label{Heisenberg_Group}             
The $n$-dimensional \emph{Heisenberg Group} $\mathbb{H}^n$ is defined as 
$
\mathbb{H}^n:= (\mathbb{R}^{2n+1}, * ),
$ 
where $*$ is the product
$$
(x,y,t)*(x',y',t') := \left  (x+x',y+y', t+t'- \frac{1}{2} \langle J
 \begin{pmatrix} 
x \\
y
\end{pmatrix} 
, 
 \begin{pmatrix} 
x' \\
y' 
\end{pmatrix} \rangle_{\mathbb{R}^{2n}} \right  ),
$$
with $x,y,x',y' \in \mathbb{R}^n$, $t,t' \in \mathbb{R}$ and $J=  \begin{pmatrix} 
 0 &  I_n \\
-I_n & 0
\end{pmatrix} $. 
It is common to write $x=(x_1,\dots,x_n) \in \mathbb{R}^n$. Furthermore, with a simple computation of the matrix product, we immediately have that
$$
(x,y,t)*(x',y',t') := \left  (x+x',y+y', t+t' + \frac{1}{2} \sum_{j=1}^n \left ( x_j y_j'  -  y_j x_j'  \right ) \right  ).
$$
\end{defin}

\noindent
One can verify that the Heisenberg group $\mathbb{H}^n$ is a Lie group, meaning that the internal operations of product and inverse are both differentiable.\\
In the Heisenberg group $\mathbb{H}^n$ there are two important groups of automorphisms; the first one is the left translation 
\begin{align*}
\tau_q : \mathbb{H}^n  \to \mathbb{H}^n, \ p \mapsto q*p,
\end{align*}
and the second one is the ($1$-parameter) group of the \emph{anisotropic dilations $\delta_r$}, with $r>0$:
\begin{align*}
\delta_r : \mathbb{H}^n   \to \mathbb{H}^n, \ (x,y,t)  \mapsto (rx,ry,r^2 t).
\end{align*}


\subsection{Left Invariance and Horizontal Structure on $\mathbb{H}^n$}\label{lefthor}

The standard basis of vector fields in the Heisenberg group $\mathbb{H}^n$ gives it the structure of Carnot group. By duality, we also introduce its standard basis of differential forms.

\begin{defin}
\label{XYT}
The standard basis of left invariant vector fields in $\mathbb{H}^n$
consists of the following: 
 $$
\begin{cases}
X_j &:= \partial_{x_j} - \frac{1}{2} y_j\partial_{t} \quad \emph{\emph{ for }}  j=1,\dots,n , \\
Y_j &:= \partial_{y_j} + \frac{1}{2} x_j\partial_{t} \quad \emph{\emph{ for }}  j=1,\dots,n,  \\
T &:= \partial_{t}.
\end{cases}
 $$
\end{defin}

\noindent
One can observe that $\{ X_1,\dots,X_n,Y_1,\dots,Y_n,T \}$ becomes $\{ \partial_{x_1},\dots, \partial_{x_n}, \partial_{y_1},\dots,\partial_{y_n}, \partial_{t} \}$ at the neutral element. Another easy observation is that the only non-trivial commutators of the vector fields $X_j,Y_j$ and $T$ are
$$
[X_j,Y_j]=T  \quad\emph{\emph{ for }} j=1,\dots,n.
$$
This immediately tells that all the higher-order commutators are zero and that the Heisenberg group is a Carnot group of step $2$. Indeed we can write its Lie algebra $\mathfrak{h}$ as 
$
\mathfrak{h} =\mathfrak{h}_1 \oplus \mathfrak{h}_2,
$ 
with
$$
\mathfrak{h}_1 = \spn \{ X_1,  \ldots, X_n, Y_1, \ldots, Y_n \} \quad \text{and} \quad \mathfrak{h}_2 =\spn \{ T \}.
$$
Conventionally one calls $\mathfrak{h}_1$ the space of \emph{horizontal} and $\mathfrak{h}_2$ the space of \emph{vertical vector fields}.

\noindent
The vector fields $\{ X_1,\dots,X_n,Y_1,\dots,Y_n\}$ are homogeneous of order $1$ with respect to the dilation $\delta_r, \  r \in \mathbb{R}^+$, i.e.,
$$
X_j (f\circ \delta_r)=r X_j(f)\circ \delta_r \quad  \text{ and }   \quad   Y_j (f\circ \delta_r)=r Y_j(f)\circ \delta_r ,
$$
where $f \in C^1 (U, \mathbb{R} )$, $U\subseteq \mathbb{H}^n$ open and $j=1,\dots,n$. On the other hand, the vector field $T$ is homogeneous of order $2$, i.e.,
$$
T(f\circ \delta_r)=r^2T(f)\circ \delta_r.
$$
It is not a surprise, then, that the homogeneous dimension of $\mathbb{H}^n$ is $Q=2n+2.$

\noindent
The vector fields $X_1,\dots,X_n,Y_1,\dots,Y_n,T$ form an orthonormal basis of $\mathfrak{h}$ with a scalar product $\langle \cdot , \cdot \rangle $. In the same way, $X_1,\dots,X_n,Y_1,\dots,Y_n$ form an orthonormal basis of $\mathfrak{h}_1$ with a scalar product $\langle \cdot , \cdot \rangle_H $ defined purely on $\mathfrak{h}_1$.\\

\begin{defin}
\label{dual_basis}
Consider the dual space of $\mathfrak{h}$, $ {\prescript{}{}\bigwedge}^1 \mathfrak{h}$, which inherits an inner product from $\mathfrak{h}$. By duality, one can find a dual orthonormal basis of covector fields $\{\omega_1,\dots,\omega_{2n+1}\}$ in $ {\prescript{}{}\bigwedge}^1 \mathfrak{h}$ such that 
$$
\langle \omega_j \vert W_k \rangle =
\delta_{jk}, \quad \text{for } j,k=1,\dots,2n+1,
$$
where $W_k$ is an element of the basis of $\mathfrak{h}$. 
Such covector fields are differential forms in the Heisenberg group.
\end{defin}

\noindent
The orthonormal basis of $ {\prescript{}{}\bigwedge}^1 \mathfrak{h}$ is given by  
$
\{dx_1,\dots,dx_n,dy_1,\dots,dy_n,\theta \},
$  
where $\theta$ is called \emph{contact form} and is defined as
$$
\theta :=dt - \frac{1}{2}  \sum_{j=1}^{n} (x_j d y_j-y_j d x_j).
$$

\begin{defin} \label{kdim}   
We define the sets of $k$-dimensional vector fields and differential forms, respectively, as:
\begin{align*}
\Omega_k \equiv {\prescript{}{}\bigwedge}_k \mathfrak{h} &:= \spn \{ W_{i_1} \wedge \dots \wedge W_{i_k} \}_{1\leq i_1 \leq \dots \leq i_k \leq 2n+1 },
\end{align*}
and
\begin{align*}
\Omega^k \equiv {\prescript{}{}\bigwedge}^k \mathfrak{h} &:= \spn \{ \theta_{i_1} \wedge \dots \wedge \theta_{i_k} \}_{1\leq i_1 \leq \dots \leq i_k \leq 2n+1 },
\end{align*}
where $W_{i_l}$'s are elements of the standard basis of $\mathfrak{h}$ and $\theta_{i_l}$'s are elements of the standard basis of $ {\prescript{}{}\bigwedge}^1 \mathfrak{h}$.\\ 
The same definitions can be given for $ \mathfrak{h}_1$ and produce the spaces $ {\prescript{}{}\bigwedge}_k \mathfrak{h}_1 $ and $ {\prescript{}{}\bigwedge}^k \mathfrak{h}_1 $.
\end{defin}

\begin{defin}
Consider a form $\omega \in  {\prescript{}{}\bigwedge}^k \mathfrak{h}$, with $k=1,\dots,2n+1$. We define $\omega^* \in  {\prescript{}{}\bigwedge}_k \mathfrak{h}$ so that
$$
\langle \omega^* , V   \rangle   =     \langle \omega \vert V   \rangle \quad \text{for all } V \in  {\prescript{}{}\bigwedge}_k \mathfrak{h}.
$$
\end{defin}

\noindent
Next we give the definition of Pansu differentiability for maps between Carnot groups $\mathbb{G}$ and $\mathbb{G}'$. After that, we state it in the special case of $\mathbb{G}=\mathbb{H}^n$ and $\mathbb{G}'=\mathbb{R}$.\\
A \emph{Carnot group} is a simply connected nilpotent Lie group and we call a function $h : (\mathbb{G},*,\delta) \to (\mathbb{G}',*',\delta')$ \emph{homogeneous} if $h(\delta_r(p))= \delta'_r \left ( h(p) \right )$ for all $r>0$.

\begin{defin}[see \cite{PANSU} and 2.10 in \cite{FSSC}]\label{dGGG}
Consider two Carnot groups $(\mathbb{G},*,\delta)$ and $(\mathbb{G}',*',\delta')$. A function $f: U \to \mathbb{G}'$, $U \subseteq \mathbb{G}$ open, is \emph{P-differentiable} at $p_0 \in U$ if there is a (unique) homogeneous Lie group 
 homomorphism $d_H f_{p_0} : \mathbb{G} \to \mathbb{G}'$ such that
$$
d_H f_{p_0} (p) := \lim\limits_{r \to 0} \delta'_{\frac{1}{r}} \left ( f(p_0)^{-1} *' f(p_0* \delta_r (p) ) \right ),
$$
uniformly for $p$ in compact subsets of $U$.
\end{defin}

\begin{defin}\label{dHHH}
Consider a function $f: U \to \mathbb{R}$, $U \subseteq \mathbb{H}^n$ open. $f$ is \emph{P-differentiable} at $p_0 \in U$ if there is a (unique) homogeneous Lie group 
 homomorphism $d_H f_{p_0} : \mathbb{H}^n \to \mathbb{R}$ such that
$$
d_H f_{p_0} (p) := \lim\limits_{r \to 0} \frac{  f \left (p_0* \delta_r (p) \right ) - f(p_0) }{r},
$$
uniformly for $p$ in compact subsets of $U$.
\end{defin}

\noindent
Consider again a function $f:U \to \mathbb{H}^n$, $U\subseteq \mathbb{H}^n$ open, and interpret $\mathbb{H}^n = \mathbb{R}^{2n+1}$ and $f$  in components as $f=(f^1,\dots,f^{2n+1})$,  $f^j:U \to \mathbb{R}$, $j=1,\dots,2n+1$. A straightforward computation shows that, if $f$ is P-differentiable in the sense of Definition \ref{dGGG}, then $f^1,\dots,f^{2n}$ are P-differentiable in the sense of Definition \ref{dHHH}.

\begin{defin}[see 2.11 in \cite{FSSC}]\label{veryfirstnabla}
Consider a function $f$ P-differentiable  at $p \in U$, $f:U \to \mathbb{R}$, $U\subseteq \mathbb{H}^n$ open. 
The \emph{Heisenberg gradient} or \emph{horizontal gradient} of $f$ at $p$ is defined as
$$
\nabla_\mathbb{H} f(p) := \left ( d_H f_p \right )^* \in \mathfrak{h}_1,
$$
or, equivalently,
$$
\nabla_\mathbb{H} f(p) = \sum_{j=1}^{n} \left [  (X_j f)(p) X_j  + (Y_j f)(p) Y_j  \right ].
$$
\end{defin}


\begin{no}[see 2.12 in \cite{FSSC}]\label{CH1}
Consider $ U \subseteq \mathbb{H}^n$ open, we say that $C_{\mathbb{H}}^1 (U, \mathbb{R})$ is the vector space of continuous functions $f:U \to \mathbb{R} $  such that $\nabla_\mathbb{H} f$ is continuous in $U$ or, equivalently, such that the P-differential $d_H f$ is continuous.
\end{no}

\noindent
To conclude this part, we define the Hodge operator which, given a vector field, returns a second one of dual dimension  
and orthogonal to the first.

\begin{defin}[see 2.3 in \cite{FSSC} or 1.7.8 in \cite{FED}]\label{hodge}
Consider $1 \leq k \leq 2n$. The \emph{Hodge operator} is the linear isomorphism
\begin{align*}
*: {\prescript{}{}\bigwedge}_k \mathfrak{h} &\rightarrow {\prescript{}{}\bigwedge}_{2n+1-k} \mathfrak{h} ,\\
\sum_I v_I V_I &\mapsto  \sum_I v_I (*V_I),
\end{align*}
where 
$
*V_I:=(-1)^{\sigma(I) }V_{I^*},
$ 
and, for $1 \leq  i_1 \leq \cdots \leq i_k \leq 2n+1$,
\begin{itemize}
\item $I=\{ i_1,\cdots,i_k \}$,
\item $V_I= V_{i_1} \wedge \cdots \wedge V_{i_k} $,
\item $I^*=\{ i_1^*,\dots,i_{2n+1-k}^* \}=\{1, \cdots, 2n+1\} \smallsetminus I \quad $  and
\item $\sigma(I)$ is the number of couples $(i_h,i_l^*)$ with $i_h > i_l^*$.
\end{itemize}
\end{defin}


\subsection{Distances and Dimensions on $\mathbb{H}^n$}

On the Heisenberg group  $\mathbb{H}^n$ we can define different equivalent distances. Then we can look at its topology and different dimensions.

\begin{defin}
\label{norm}
We define the \emph{Kor\'anyi} distance on $\mathbb{H}^n$ by setting, for $p,q \in \mathbb{H}^n$,
$$
d_{\mathbb{H}} (p,q) :=  \norm{ q^{-1}*p }_{\mathbb{H}},
$$
where $ \norm{ \cdot }_{\mathbb{H}}$ is the \emph{Kor\'anyi}  norm
$$
\norm{(x,y,t)}_{\mathbb{H}}:=\left (  |(x,y)|^4+16t^2  \right )^{\frac{1}{4}},
$$
with $(x,y,t) \in \mathbb{R}^{2n} \times  \mathbb{R} $ and $| \cdot |$ being the Euclidean norm.
\end{defin}

\noindent
The Kor\'anyi distance is left invariant, i.e,
$$
d_{\mathbb{H}} (p*q,p*q')=d_{\mathbb{H}} (q,q'),  \quad  p,q,q' \in \mathbb{H}^n,
$$
and homogeneous of degree $1$ with respect to $\delta_r$, i.e,
$$
d_\mathbb{H} \left ( \delta_r (p), \delta_r (q)  \right ) = r d_{\mathbb{H}} (p,q) ,  \quad  p,q \in \mathbb{H}^n, \quad r>0.
$$

\noindent
Furthermore, the \emph{Kor\'anyi} distance is equivalent to the \emph{Carnot--Carath\'eodory} distance $d_{cc}$, which is defined as the infimum of all lengths of curves between two points whose tangent vector fields are horizontal.

\begin{no}\label{CCK}
Consider a surface $S \subseteq \mathbb{H}^n$. We denote the Hausdorff dimension of $S$ with respect to the Euclidean distance as 
$$
\dim_{\mathcal{H}_{E}} S,
$$ 
while its Hausdorff dimension with respect to the Carnot--Carath\'eodory and Kor\'anyi distances as
$$
\dim_{\mathcal{H}_{cc}} S= \dim_{\mathcal{H}_{\mathbb{H}}} S.
$$
\end{no}


\section{Orientability} \label{orient4}
In this section we first discuss the notion of $\mathbb{H}$-regularity for low dimension and low codimension surfaces in the Heisenberg group; this work was inspired by the research of Bruno Franchi, Raul Serapioni and Francesco Serra Cassano \cite{FSSC}. Then we analyse the behaviour of a M{\"o}bius Strip in $\mathbb{H}^1$ and find a  $1$-codimensional $C^1$-Euclidean subset with no characteristic points and non-orientable in the Euclidean sense. This subset is, in particular, a $1$-codimensional $\mathbb{H}$-regular surface non-orientable in the Euclidean sense. \\
Next, we introduce and characterise the notion of orientability in the Heisenberg sense ($\mathbb{H}$-orientability), which, as one would expect, is invariant under left translations and anisotropic dilations for $\mathbb{H}$-regular $1$-codimensional surfaces. Lastly, we show that, for regular enough surfaces, $\mathbb{H}$-orientability implies Euclidean-orientability. As a consequence, we conclude that non-$\mathbb{H}$-orientable $\mathbb{H}$-regular surfaces exist in $\mathbb{H}^1$.

\subsection{$\mathbb{H}$-regularity in $\mathbb{H}^n$}\label{subsec:Hreg}

We state here the definitions of $\mathbb{H}$-regularity for low dimension and low codimension. Then we proceed to define normal and tangent vector fields and characteristic points.

\begin{defin}[see 3.1 in \cite{FSSC}]
Consider $1\leq k \leq n$. A subset $S \subseteq \mathbb{H}^n$ is a $\mathbb{H}$-\emph{regular} $k$-\emph{dimensional surface} if for all $p \in S$   there exists  a neighbourhood $U$ of $p$,   
 an open set $ V \subseteq \mathbb{R}^k$ and a function $\varphi : V \to U$, $ \varphi \in C_{\mathbb{H}}^1(V,U) $  injective with $d_H \varphi $ injective, such that $ S \cap U = \varphi (V)$.
\end{defin}

\begin{defin}[see 3.2 in \cite{FSSC}]\label{Hreg}
Consider $1\leq k \leq n$. A subset $S \subseteq \mathbb{H}^n$ is a $\mathbb{H}$-\emph{regular} $k$-\emph{codimensional surface} if for all $ p \in S $ there exists a neighbourhood  $U$ of $p$   
 and a function  $ f : U \to \mathbb{R}^k$, $ f \in C_{\mathbb{H}}^1(U,\mathbb{R}^k)$, such that  $  {\nabla_\mathbb{H} f_1} \wedge \dots \wedge {\nabla_\mathbb{H} f_k}   \neq 0 $ on   $ U $ and  $  S \cap U = \{ f=0 \} $.
\end{defin}

\noindent
We will almost always work with the codimensional definition, that is, the surfaces of higher dimension.  
If a surface is $\mathbb{H}$-regular, it is natural to associate to it, locally, a normal and a tanget vector field:

\begin{defin}\label{def:hornormal}
Consider a $\mathbb{H}$-regular $k$-codimensional surface $S$ and $p \in S$. Then the \emph{(horizontal) normal $k$-vector field} $n_{\mathbb{H},p}$ is defined as
$$
n_{\mathbb{H},p} := \frac{ {\nabla_\mathbb{H} f_1}_p \wedge \dots \wedge {\nabla_\mathbb{H} f_k}_p   }{ \vert {\nabla_\mathbb{H} f_1}_p \wedge \dots \wedge {\nabla_\mathbb{H} f_k}_p \vert }    \in {\prescript{}{}\bigwedge}_{k,p} \mathfrak{h}_1 .
$$
In a natural way, the \emph{tangent $(2n+1-k)$-vector field} $t_{\mathbb{H},p}$ is defined as the dual of $n_{\mathbb{H},p}$:
$$
t_{\mathbb{H},p} := * n_{\mathbb{H},p} \in {\prescript{}{}\bigwedge}_{2n+1-k,p} \mathfrak{h},
$$
where $*$ is the Hodge operator of Definition \ref{hodge}. If $n_{\mathbb{H},p} $ and $t_{\mathbb{H},p} $ can be defined globally, then they are denoted $n_{\mathbb{H}} $ and $t_{\mathbb{H}} $.
\end{defin}

\noindent
When considering the regularity of a surface in the Euclidean sense, as opposed to the $\mathbb{H}$-regularity, we say $C^k$\emph{-regular in the Euclidean sense} or $C^k$\emph{-Euclidean} for short. This obviously means that we are looking at $S\subseteq \mathbb{H}^n$ as a subset of $\mathbb{R}^{2n+1}$.


\begin{lem}\label{quickcomputation}
Consider a $C^1$-Euclidean surface   $S$  in $\mathbb{H}^n$. Then $\dim_{\mathcal{H}_{cc}} S=Q-1=2n+1$ if and only if $\dim_{\mathcal{H}_{E}} S =2n$.
\end{lem}

\noindent
Hence, from now on, we may consider a $C^1$-Euclidean $1$-codimensional surface without further specifications on the dimension.

\begin{proof}
Consider first $\dim_{\mathcal{H}_{cc}} S=2n+1$. The Hausdorff dimension of $S$ with respect to the Euclidean distance is equal to the dimension of the tangent plane, which is well defined everywhere by hypothesis; hence such dimension is an integer. By theorems $2.4$-$2.6$ in \cite{BTW} or by \cite{BRSC} (for $\mathbb{H}^1$ only) and with $k=\dim_{\mathcal{H}_{E}} S$, one has that:
$$
\max \{  k, 2k-2n \} \leq \dim_{\mathcal{H}_{cc}} S \leq \min \{ 2k, k+1 \} .
$$
The second inequality says that 
$$
2n+1 = \dim_{\mathcal{H}_{cc}} S \leq k+1,
$$
meaning $2n \leq k$. Then the only possible cases are $k=2n$ and $k=2n+1$.\\
Next, 
 if $k <  2k-2n$ (so $k > 2n$), on one side the only possibility becomes $k=2n+1$. On the other side, $k$ is also strictly less than $\max \{  k, 2k-2n \} =2k-2n$, which must be less than or equal to $\dim_{\mathcal{H}_{cc}} S$: 
$$
2n+1 = k < \max \{  k, 2k-2n \} \leq \dim_{\mathcal{H}_{cc}} S =2n+1,
$$
which is impossible. Then $2k-2n\leq k$, meaning $k\leq 2n$.  
So the only possibility is $k=2n$.\\
On the other hand, if we consider a $C^1$-Euclidean surface $S \subseteq \mathbb{R}^{2n+1}$ with $\dim_{\mathcal{H}_{E}} S =2n$ (an hypersurface in the Euclidean sense), then it follows (see page 64 in \cite{BAL} or by \cite{GROMOV}) that $\dim_{\mathcal{H}_{cc}} S=2n+1$.
\end{proof}

\begin{defin}
Consider a surface $S \subseteq \mathbb{H}^n$ and denote $ T_p S $ the space of  vectors tangent to $S$ at the point $p$. Define the \emph{characteristic set} $C(S)$ of $S$ as 
$$
C(S):= \left \{   p \in S ; \ T_p S \subseteq \mathfrak{h}_{1,p}   \right \}.
$$
This says that a point $p \in C(S)$ if and only if $n_{\mathbb{H},p}=0$. For the $k$-codimensional case, this also means that it is not possible to find a map $f$ such as in Definition \ref{Hreg}. 
\end{defin}

\noindent
By 1.1 in \cite{BAL} and 2.16 in \cite{MAG} we can infer that the set of characteristic points of a $k$-codimensional surface $S \subseteq \mathbb{H}^n$ has always measure zero:
$$
\mathcal{H}_{cc}^{2n+2-k} \left (  C(S)   \right ) =0.
$$
Furthermore, from page $195$ in \cite{FSSC}, we can say that a $C^1$-Euclidean surface $S$ with $C(S)= \varnothing$ is a $\mathbb{H}$-regular surface in $\mathbb{H}^n$.


\subsection{The M{\"o}bius Strip in $\mathbb{H}^1$}\label{subsec:Mobius}

In this subsection we show that, at least when $n=1$, there exist $1$-codimensional $C^1$-Euclidean surfaces with no characteristic points that are non-orientable in the Euclidean sense. This implies that there exist $1$-codimensional $\mathbb{H}$-regular surfaces which are non-orientable in the Euclidean sense. \\

\noindent
We prove this by considering a M\"obius strip, the most classical non-orientable object in the Euclidean sense. We will define only later (Definition \ref{Eorientable}) the notion of orientability in the Euclidean sense but here we only need the knowledge that the M\"obius strip is not Euclidean-orientable.

\noindent
Let $\mathcal{M}$ be any M\"obius strip. Is $C(\mathcal{M}) = \varnothing$? Or, if not, is there a surface $\widetilde{\mathcal{M}}  \subseteq \mathcal{M}$, $\widetilde{\mathcal{M}}$ still non-orientable in the Euclidean sense, such that  $C( \widetilde{\mathcal{M}} ) = \varnothing$? We will attempt an answer by considering one specific parametrisation of the M\"obius strip.

\noindent
Consider the fixed numbers $R  \in \mathbb{R}^+ $ and $w \in \mathbb{R}^+$ so that $w<R$. Then consider the map
\begin{align*}
\gamma :& [0,2 \pi ) \times [-w,w] \to  \mathbb{R}^3  \\
\gamma (r,s):&=(x(r,s), y(r,s), t(r,s))\\
&= \left (  \left [ R+s \cos \left  ( \frac{r}{2} \right ) \right ] \cos r  , \  \left [R+s \cos \left ( \frac{r}{2} \right ) \right ] \sin r , \  s \sin \left  ( \frac{r}{2} \right ) \right  ).
\end{align*}
This is a parametrisation of a M{\"o}bius strip of half-width $w$ with midcircle of radius $R$  in $\mathbb{R}^3$. We can denote then $\mathcal{M} := \gamma \left  ( [0,2 \pi ) \times [-w,w] \right  ) \subseteq  \mathbb{R}^3$.

\begin{prop}\label{Mobius}
Consider the M\"obius strip $\mathcal{M}$ parametrised by the curve $\gamma$. Then $\mathcal{M}$ contains at most one characteristic point $\tilde p$ and so there exists a $1$-codimensional  $C^1$-Euclidean surface $\widetilde{\mathcal{M}}  \subseteq \mathcal{M}$ such that $\tilde p  \notin \widetilde{\mathcal{M}}$, $\widetilde{\mathcal{M}}$ still non-orientable in the Euclidean sense and  $C( \widetilde{\mathcal{M}} ) = \varnothing$.
\end{prop}

\noindent
This  says, by our discussion at the end of Subsection \ref{subsec:Hreg}, that $\widetilde{\mathcal{M}}$ is a $\mathbb{H}$-regular surface and is non-orientable in the Euclidean sense. The proof will follow after some lemmas (which are proved in E.1-E.5 in \cite{GClicentiate}).

\begin{lem}[Step 1]
Consider the parametrisation $\gamma$. The two tangent vector fields of $\gamma$, in the basis $\{\partial_x, \partial_y, \partial_t \}$, are
\begin{align*}
\vec\gamma_r (r,s) = & \bigg (     - \frac{s}{2} \sin \left  ( \frac{r}{2} \right )     \cos r -  \left [ R+s \cos \left  ( \frac{r}{2} \right ) \right ]     \sin r     , \\ 
& - \frac{s}{2} \sin \left  ( \frac{r}{2} \right )     \sin r +  \left [ R+s \cos \left  ( \frac{r}{2} \right ) \right ]   \cos r,  \     \frac{s}{2} \cos \left  ( \frac{r}{2} \right )    \bigg  ),\\
\text{and}\hspace{0.8cm} &\\
\vec\gamma_s (r,s) =& \left (   \cos \left  ( \frac{r}{2} \right )  \cos r   , \    \cos \left ( \frac{r}{2} \right )  \sin r  , \    \sin \left  ( \frac{r}{2} \right ) \right  ).
\end{align*}
\end{lem}

\begin{lem}[Step 2]
Consider the parametrisation $\gamma$. The two tangent vector fields $\vec\gamma_r$ and $\vec\gamma_s$ can be written in Heisenberg coordinates as:
\begin{align*}
\vec\gamma_r (r,s)= & \left (   - \frac{1}{2}s \sin \left  ( \frac{r}{2} \right )     \cos r -  \left [ R+s \cos \left  ( \frac{r}{2} \right ) \right ] \sin r  \right ) X \\
&+ \bigg ( - \frac{1}{2}s \sin \left  ( \frac{r}{2} \right )     \sin r + \left [ R+s \cos \left  ( \frac{r}{2} \right ) \right ]  \cos r \bigg ) Y\\
& +  \left  (  s \frac{1}{2} \cos \left  ( \frac{r}{2} \right )  - \left [ R+s \cos \left  ( \frac{r}{2} \right ) \right ]^2 \frac{1}{2} \right )T,\\
\text{and}\hspace{0.8cm} &\\
\vec\gamma_s (r,s) = &     \cos \left  ( \frac{r}{2} \right )  \cos r X   +  \cos \left ( \frac{r}{2} \right )  \sin r Y  +  \sin \left  ( \frac{r}{2} \right )  T .     
\end{align*}
\end{lem}

\noindent
Call $\vec{N} (r,s)= \vec{N}_1  (r,s) X + \vec{N}_2  (r,s) Y + \vec{N}_3  (r,s) T$ the normal vector field of $\mathcal{M}$. Such vector is given by the cross product of the two tangent vector fields $\vec\gamma_r$ and $\vec\gamma_s$. Specifically:
$$
\vec{N}= \vec{N}_1 X + \vec{N}_2 Y + \vec{N}_3 T =
\vec\gamma_r \times_\mathbb{H}  \vec\gamma_s =
$$ 
\[ =
\begin{vmatrix}
X & Y & T \\ 
    \frac{-s  \cos r \sin  \frac{r}{2} }{2}   -   \left [ R+s \cos  \frac{r}{2}  \right ] \sin r&
 \frac{-s  \sin r \sin  \frac{r}{2}  }{2}   +  \left [ R+s \cos \frac{r}{2}  \right ]  \cos r  &
  \frac{s  \cos  \frac{r}{2}  }{2}    -  \frac{   \left [ R+s \cos \frac{r}{2} \right ]^2   }{2}\\
\cos \left  ( \frac{r}{2} \right )  \cos r &  \cos \left ( \frac{r}{2} \right )  \sin r &  \sin \left  ( \frac{r}{2} \right )
\end{vmatrix}.
\]

\begin{lem}[Step 3]
Consider  the normal vector field of $\mathcal{M}$, $\vec{N} (r,s)= \vec{N}_1  (r,s) X + \vec{N}_2  (r,s) Y + \vec{N}_3  (r,s) T$. A computation shows that:
\begin{align*}
\vec{N}_1 (r,s)=&
 - \frac{1}{2}s    \sin r 
+  \left [ R+s \cos \left  ( \frac{r}{2} \right ) \right ]  \cos r  \sin \left  ( \frac{r}{2} \right )     
 +  \left [ R+s \cos \left  ( \frac{r}{2} \right ) \right ]^2 \frac{1}{2}  \cos \left ( \frac{r}{2} \right )  \sin r  ,\\
\vec{N}_2 (r,s)=&
\left (   - z^5     +  \frac{1}{2}  z^3   \right )   s^2   
+   \left (    - 2 (R+1)z^4       +    ( R+3) z^2     - \frac{1}{2}   \right )   s   
   - ( R^2  + 2   R) z^3  \\
& + \left ( \frac{1}{2} R^2    +  2   R  \right  )   z   \quad \quad \quad \quad \text{and}\\
 \vec{N}_3 (r,s)=& -   \left [ R+s \cos \left  ( \frac{r}{2} \right ) \right ]   \cos \left ( \frac{r}{2} \right ) .
\end{align*}
with $z=\cos \left  ( \frac{r}{2} \right ) $, $r\in [0,2\pi)$ and $s \in [-w,w]$. 
\end{lem}

\begin{proof}[Proof of Proposition \ref{Mobius}] 
To find  pairs of parameters $(r,s)$ corresponding to characteristic points we have to impose
$$
\begin{cases}
\vec{N}_1  (r,s)= 0,\\
\vec{N}_2  (r,s)= 0.
\end{cases}
$$
A computation shows that $\vec{N}_1  (r,s)= 0$ only at the points $(x(r,s),y(r,s),t(r,s))$ with
$$
(r,s)=
\begin{cases}
(0,s), \quad  &s \in [-w,w], \quad \text{or} \\
\left (r, \frac{ -(R+1) z^2 + 1 \pm \sqrt{ z^4  -( R+2 ) z^2 +1 } }{ z^3} \right  ),
 \quad &r \in [0, 2 \pi), \ r \neq \pi, \ z=\cos \frac{r}{2}.
\end{cases}
$$
Evaluating these possibilities on $\vec{N}_2  (r,s)= 0$, another computation (see E.6 and E.7 in \cite{GClicentiate}) shows that the system $\{ \vec{N}_1  (r,s)= 0; \
\vec{N}_2  (r,s)= 0 \}$ is verified only by the pair  
$$
(r,s)= \left (0,\frac{ -2R+1 - \sqrt{-4R+1}  }{2} \right ),\quad \text{when} \quad 0< R < \frac{1}{4},
$$
which corresponds to the point $\tilde p = (\bar x,\bar y,\bar t) = \left ( \frac{1}{2} - \sqrt{-R + \frac{1}{4} } ,0,0 \right ) $:
$$
\begin{cases}
\bar x	=	[R+s \cos ( \frac{r}{2} ) ] \cos r =R +\frac{ -2R+1 - \sqrt{-4R+1}  }{2} = \frac{ 1 - \sqrt{-4R+1}  }{2} =\frac{1}{2} - \sqrt{-R + \frac{1}{4} }>0	\\
\bar y	=	[R+s \cos ( \frac{r}{2} ) ] \sin r=0	\\
\bar t	=	s \sin ( \frac{r}{2} )=0.	
\end{cases}
$$
This is a characteristic point.  
Notice that it is not strange that the number of characteristic points depends on the radius $R$, as changing the radius is not an anisotropic dilation. Therefore the surface
$$
\widetilde{\mathcal{M}} :=\mathcal{M}  \ \setminus \  U_{\tilde p},
$$
where $U_{\tilde p}$ is a neighbourhood of $\tilde p$ with smooth boundary, is indeed a $C^1$-Euclidean surface with $C(\widetilde{\mathcal{M}}) = \varnothing$, hence $1$-codimensional $\mathbb{H}$-regular, and not Euclidean-orientable. This completes the proof. 
\end{proof}


\subsection{Comparing Orientabilities}

In this section we first recall the definition of Euclidean-orientability and introduce and characterise the notion of orientability in the Heisenberg sense ($\mathbb{H}$-orientability). Next we prove that, under left translations and anisotropic dilations, $\mathbb{H}$-regularity is invariant for $1$-codimensional surfaces and $\mathbb{H}$-orientability is invariant for $\mathbb{H}$-regular $1$-codimensional surfaces. Lastly, we show how the two notions of orientability are related, concluding that, for regular enough surfaces, $\mathbb{H}$-orientability implies Euclidean-orientability. This allows us to conclude that non-$\mathbb{H}$-orientable $\mathbb{H}$-regular surfaces exist, at least when $n=1$.\\

\noindent
Recall that, by Definition \ref{Hreg}, $S$ is a $\mathbb{H}$-regular $1$-codimensional surface in $\mathbb{H}^n$ if:
\begin{equation}\label{eq1}
\text{for all } p \in S \ \text{ there exists a neighbourhood }  U  \text{ and }  f : U \to \mathbb{R}, \ f \in C_{\mathbb{H}}^1 (U, \mathbb{R}),  \text{ so that } 
\end{equation}
$$
S \cap U = \{ f=0 \} \text{ and } \nabla_{\mathbb{H}} f \neq 0 \text{ on } U.
$$
On the other hand, if $S$ is $C^1$-Euclidean, then (see for instance the introduction of \cite{BAL}):
\begin{equation}\label{eq2}
\text{for all } p \in S \ \text{ there exists a neighbourhood }  U \text{ and } g : U \to \mathbb{R}, \ g \in C^1(U, \mathbb{R}), \text{ so that } 
\end{equation}
$$
S \cap U = \{ g=0 \} \text{ and } \nabla g \neq 0 \text{ on } U.
$$
These two notions of regularity are obviously similar. Next we connect each of them to a definition of orientability, which we then compare.


\subsubsection{$\mathbb{H}$-Orientability in $\mathbb{H}^n$}\label{Horient}$\\$

%
\noindent
Consider a surface $S \subseteq \mathbb{H}^n$ and the space of  vector fields tangent to $S$, $ T S $. A vector $v$ is \emph{normal to $S$}, $v \perp S$, if $\langle v,w \rangle =0$ for all $w \in TS.$

\begin{defin}\label{Eorientable}
Consider a $1$-codimensional $C^1$-Euclidean surface $S \subseteq \mathbb{H}^n$ with $C(S) = \varnothing$. The surface $S$ is \emph{Euclidean-orientable} (or \emph{orientable in the Euclidean sense}) if there exists a continuous global $1$-vector field 
$$
 n_E=\sum_{i=1}^{n} \left ( n_{E,i} \partial_{x_i} + n_{E,n+i} \partial_{y_i} \right ) + n_{E,2n+1} \partial_t  \neq 0 ,
$$
defined on $S$ and normal to $S$. Such $n_{E}$ is called \emph{Euclidean normal vector field of} $S$.
\end{defin}

\noindent
Equivalently, the surface $S$ is Euclidean-orientable if there exists a continuous global $2n$-vector field $t_E$ on $S$, so that $t_E$ is tangent to $S$. This is the same as saying that $*t_E$ is normal to $S$, where $*$ is the Hodge operator (see Definition \ref{hodge}). It is also straightforward that, up to a choice of sign, $t_E = * n_E.$

\begin{defin}
Consider two vectors $v, w \in \mathfrak{h}_1 
$ in $\mathbb{H}^n$; they are \emph{orthogonal in the Heisenberg sense}, $v \perp_H  w$,  if 
$$
 \langle v,w \rangle_H = 0,
$$
where $\langle \cdot , \cdot \rangle_H$ is the scalar product that makes $X_j$'s and $Y_j$'s orthonormal.
\end{defin}

\begin{defin}
Consider a $1$-codimensional $C^1$-Euclidean surface $S \subseteq \mathbb{H}^n$ with $C(S) = \varnothing$. Consider also a vector $v \in  \mathfrak{h}_1$. We say that $v$ and $S$ are $\mathbb{H}$-orthogonal (\emph{orthogonal in the Heisenberg sense}), and we write $v \perp_H  S$, if 
$$
 \langle v,  w_{\vert_{ \mathfrak{h}_1 }}  \rangle_H = 0, \quad \text{for all } \ w \in TS.
$$
In the same way, we say that a $2n$-vector field $v \in {\prescript{}{}\bigwedge}_{2n} \mathfrak{h}$ 
 is $\mathbb{H}$-tangent (\emph{tangent to} $S$ \emph{in the Heisenberg sense})  to $S$ if  $*v \in  \mathfrak{h}_1$ and
$$
\langle *v, w_{\vert_{ \mathfrak{h}_1 }} \rangle_H = 0,  \quad \text{for all } \ w \in TS.
$$
\end{defin}


\begin{defin}\label{def:Horientable}
Consider a $1$-codimensional $C^1$-Euclidean surface $S \subseteq \mathbb{H}^n$ with $C(S) = \varnothing$. We say that $S$ is $\mathbb{H}$\emph{-orientable} (or \emph{orientable in the Heisenberg sense})  if there exists a continuous 
horizontal 
global 
never-null 
$1$-vector field 
$n_{\mathbb{H}}$, i.e., 
$$
n_{\mathbb{H}}=\sum_{i=1}^{n} \left ( n_{\mathbb{H},i} X_i + n_{\mathbb{H},n+i} Y_i \right ) \neq 0,
$$
defined on $S$ so that $n_{\mathbb{H}}$ and $S$ are $\mathbb{H}$-orthogonal. 
Note that $n_{\mathbb{H}}$ is consistent with Definition \ref{def:hornormal}.
\end{defin}

\begin{lem}
Consider a $1$-codimensional $C^1$-Euclidean surface $S \subseteq \mathbb{H}^n$ with $C(S) = \varnothing$. The following are equivalent:
\begin{enumerate}[label=(\roman*)]
\item
$S$ is $\mathbb{H}$\emph{-orientable},
\item
there exists a continuous global $2n$-vector field $t_\mathbb{H}$ on $S$ so that $t_\mathbb{H}$ is $\mathbb{H}$-tangent to $S$. 
\end{enumerate}
\end{lem}

\noindent
One can easily see that, up to a choice of sign, $t_\mathbb{H}= * n_\mathbb{H}.$\\
It is possible to give an equivalent definition of orientability using differential forms by saying that a manifold is orientable (either in the Euclidean or Heisenberg sense) if and only if there exists a continuous volume form on it, where a volume form is a never-null form of maximal order.

\begin{lem}
Consider a $1$-codimensional $C^1$-Euclidean surface $S \subseteq \mathbb{H}^n$ with $C(S) = \varnothing$.
The following are equivalent:
\begin{enumerate}[label=(\roman*)]
\item
$S$ is $\mathbb{H}$-orientable,
\item
$S$ allows a continuous volume form $\omega_\mathbb{H}$ (a never-null form of maximal order), which can be chosen so that the following property holds: 
$$
\langle \omega_\mathbb{H} \vert t_\mathbb{H} \rangle=1.
$$
\end{enumerate}
\end{lem}

\noindent
Note that, if the condition is verified in $\mathbb{H}^1$, the volume form is of the kind:
$$
\omega_\mathbb{H}= \frac{ n_{\mathbb{H},1} }{ n_{\mathbb{H},1}^2+n_{\mathbb{H},2}^2} dy \wedge \theta 
- \frac{ n_{\mathbb{H},2} }{ n_{\mathbb{H},1}^2+n_{\mathbb{H},2}^2 } dx \wedge \theta.
$$
\noindent
Also note that by condition \eqref{eq1}, locally on a neighbourhood $U$ of a point $p$, $n_\mathbb{H}=\lambda \nabla_{\mathbb{H}} f$, with $\lambda \in C^\infty (\mathbb{H}^1,\mathbb{R})$ and $ f \in C_{\mathbb{H}}^1 (U, \mathbb{R})$.  So,
$$
n_{\mathbb{H}} = n_{\mathbb{H},1}  X  +n_{\mathbb{H},2}  Y  =\lambda X f X +  \lambda Y f Y
$$
and, since $t_{\mathbb{H}}  = * n_{\mathbb{H}} = n_{\mathbb{H},1} Y \wedge T - n_{\mathbb{H},2} X \wedge T $,
$$
t_{\mathbb{H}}  =    \lambda X f   Y \wedge T   - \lambda Y f  X \wedge T .
$$

\begin{ex}
Consider a $\mathbb{H}$-orientable 1-codimensional surface $S \subseteq \mathbb{H}^1$. Then at each point $p \in S$ there exist two continuous global linearly independent vector fields $ \vec{r}$ and $  \vec{s}$ tangent on $S$, $ \vec{r},  \vec{s} \in T_p S$. With the previous notation, we can explicitly find such a pair by solving the following list of conditions:
    \begin{multicols}{2}
\begin{enumerate}[nosep]
\item
$\langle \vec{r},  \vec{s} \rangle_{H} =0$,
\item
$\langle \vec{r},  n_\mathbb{H} \rangle_{H} =0$,
\item
$\langle   \vec{s}, n_\mathbb{H} \rangle_{H} =0$,
\item
$\vert \vec{r} \vert_{H} =1$,
\item
$\vert \vec{s} \vert_{H} =1$,
\item
$\vec{r} \times  \vec{s} = n_\mathbb{H}$,
\item
$\vec{r} \wedge  \vec{s} = t_\mathbb{H}.$
\end{enumerate}
    \end{multicols}
\noindent
One can (but it is not necessary) choose $\vec{r} = T$ since $n_\mathbb{H} \in \spn \{X,Y\}$. Then one can take $\vec{s} = aX+bY$, so the first two conditions are satisfied.
The third condition is $\langle  \vec{s}, n_\mathbb{H} \rangle_{H} =0$, meaning
$$
 a n_{\mathbb{H},1} + b n_{\mathbb{H},2} =0,
$$
whose solution is
$$
\begin{cases}
a= c n_{\mathbb{H},2}\\
b = -c n_{\mathbb{H},1}.
\end{cases}
$$
with $c$ arbitrary. The fourth condition is verified by our choice of $\vec{r}$.\\
We have just seen that there exists a local function $f$ so that $\vec{s} = c n_{\mathbb{H},2} X - c n_{\mathbb{H},1} Y$ becomes
$$
\vec{s} =  c  \lambda Y f X  -c \lambda X f Y.
$$
Then the fifth condition, $\vert \vec{s} \vert_{\mathbb{H}} =1$, gives
$$
1= \sqrt{c^2  \lambda^2 (Y f)^2 + c^2 \lambda^2 (X f)^2} = \vert c \lambda\vert \sqrt{ (Y f)^2 +  (X f)^2} = \vert c \lambda\vert \cdot \vert   \nabla_\mathbb{H} f   \vert,
$$
meaning
$$
 c=\pm \frac{1}{ \vert   \lambda  \nabla_\mathbb{H} f   \vert}.
$$
So one has that
\begin{align*}
\vec{s} &= \pm \frac{1}{ \vert   \lambda  \nabla_\mathbb{H} f   \vert} \left (  \lambda Y f X  - \lambda X f Y \right )=
\pm \frac{\lambda}{ \vert   \lambda  \vert }    \frac{  Y f X  -  X f Y }{  \vert  \nabla_\mathbb{H} f   \vert}\\
&=
\pm \ sign (\lambda)   \left ( 
  \frac{  Y f  }{  \vert  \nabla_\mathbb{H} f   \vert}X  -     \frac{  X f  }{  \vert  \nabla_\mathbb{H} f   \vert}Y
\right )  .
\end{align*}
The sixth condition is $\vec{r} \times  \vec{s} = n_{\mathbb{H}}$, so:
$$
 n_{\mathbb{H}}= \vec{r} \times  \vec{s} =
\begin{vmatrix}
X & Y & T \\ 
0 & 0 &  1 \\ 
c  n_{\mathbb{H},2} & -c  n_{\mathbb{H},1}  & 0
\end{vmatrix}
=
c  n_{\mathbb{H},2} Y + c  n_{\mathbb{H},1} X =c  n_{\mathbb{H}}.
$$
Then it is necessary to take $c=1$ and one has $\vert \lambda \nabla_\mathbb{H} f   \vert = 1$, namely,
$$
 \lambda = \pm  \frac{1}{\vert \nabla_\mathbb{H} f   \vert },
$$
and
$$
\vec{s} =   \lambda Y f X  - \lambda X f Y   = \pm    \left ( 
  \frac{  Y f  }{  \vert  \nabla_\mathbb{H} f   \vert}X  -     \frac{  X f  }{  \vert  \nabla_\mathbb{H} f   \vert}Y
\right )  .
$$
Finally, we verify $\vec{r} \wedge  \vec{s} =  t_{\mathbb{H}}$ (the seventh and last condition):
$$
\vec{r} \wedge \vec{s} =  T \wedge ( \lambda Y f X  - \lambda X f Y  ) =\lambda X f Y \wedge T - \lambda Y f X \wedge T= t_{\mathbb{H}}.
$$
\end{ex}


\subsubsection{Invariances}

For $1$-codimensional surfaces, the $\mathbb{H}$-regularity is invariant under left translations and anisotropic dilations. Furthermore, for $\mathbb{H}$-regular $1$-codimensional surfaces, the $\mathbb{H}$-orientability is invariant under the same two types of transformations.

\begin{prop}\label{tau_prop}
Consider the left translation map $\tau_{\bar{p}} : \mathbb{H}^n \to \mathbb{H}^n$
, $\bar{p} \in \mathbb{H}^n$ and a $\mathbb{H}$-regular $1$-codimensional surface $S \subseteq  \mathbb{H}^n$. Then $\tau_{\bar{p}} S  := \{ \bar{p}*p ; \ p \in S   \}$ is again a $\mathbb{H}$-regular $1$-codimensional surface in $\mathbb{H}^n$.
\end{prop}

\begin{proof}
Since $\tau_{\bar{p}} S = \{ \bar{p}*p ; p \in S   \}$, for all $ q \in \tau_{\bar{p}} S$ there exists a point $ p \in S$ so that $q=\bar{p}*p$. 
For such $p \in S $, there exists a neighbourhood $U_p$ and a function $f: U_p \to \mathbb{R}$ so that $S \cap U_p = \{ f=0 \}$ and $\nabla_\mathbb{H} f \neq 0$ on $U_p$. 
Define $U_q := \tau_{\bar{p}} U_p = \bar{p} * U_p $, which is a neighbourhood of $q=  \bar{p}*p$, and a function $\tilde{f}: = f \circ \tau_{\bar{p}}^{-1} : U_q \to \mathbb{R} $. 
Then, for all $ q' \in U_q$,
$$
\tilde{f}(q') = (f \circ \tau_{\bar{p}}^{-1})(q')  =  f (\bar{p}^{-1} * \bar{p}* p' ) = f(p')=0,
$$
where $q'= \bar{p}* p' $ and $p' \in U_p$. Then
$$
\tau_{\bar{p}} S \cap U_q = \{ \tilde{f}=0 \}.
$$
Furthermore, on $U_q$, and by left invariance,
\begin{align*}
\nabla_\mathbb{H} \tilde{f} = &\nabla_\mathbb{H} (  f \circ \tau_{\bar{p}}^{-1} ) =  \nabla_\mathbb{H} (  f \circ \tau_{\bar{p}^{-1}}) =
\sum_{i=1}^n \bigg (
X_i(  f \circ \tau_{\bar{p}^{-1}}) X_i + Y_i (  f \circ \tau_{\bar{p}^{-1}}) Y_i
\bigg )\\
=&
\sum_{i=1}^n \bigg (
[X_i (  f )\circ \tau_{\bar{p}^{-1}} ] X_i + [ Y_i (  f ) \circ \tau_{\bar{p}^{-1}} ] Y_i
\bigg )
 \neq 0
\end{align*}
as $X_i (  f )\circ \tau_{\bar{p}^{-1}}$ and $ Y_i (  f ) \circ \tau_{\bar{p}^{-1}}$ are defined on $U_p$ and on $U_p$ one of the two is always non-negative by the hypothesis that $\nabla_\mathbb{H} f \neq 0$ on $U_p$.
\end{proof}

\begin{prop}\label{delta_prop}
Consider the usual anisotropic dilation $\delta_r : \mathbb{H}^n \to \mathbb{H}^n$
, $r>0$, and a $\mathbb{H}$-regular $1$-codimensional surface $S \subseteq  \mathbb{H}^n$. Then $\delta_r S := \{ \delta_r(p) ; \ p \in S   \}$ is again a $\mathbb{H}$-regular $1$-codimensional surface in $\mathbb{H}^n$.
\end{prop}

\begin{proof}
Since $\delta_r S = \{ \delta_r(p) ; \ p \in S   \}$, then for all $ q \in \delta_r S$ there exists a point $  p \in S$ so that $q=\delta_r(p)$. For such $p \in S $, there exists a neighbourhood $U_p$ and a function $f: U_p \to \mathbb{R}$ so that $S \cap U_p = \{ f=0 \}$ and $\nabla_\mathbb{H} f \neq 0$ on $U_p$. 
Define $U_q := \delta_r( U_p )  $, which is a neighbourhood of $q=  \delta_r ( p )$, and a function $\tilde{f}: = f \circ \delta_{1/r} : U_q \to \mathbb{R} $. 
Then, for all $ q' \in U_q$,
$$
\tilde{f}(q') = (f \circ \delta_{1/r}   )(q')  =  f ( \delta_{1/r}  \delta_r p' ) = f(p')=0,
$$
where $q'= \bar{p}* p' $ and $p' \in U_p$. Then
$$
 \delta_r S \cap U_q = \{ \tilde{f}=0 \}.
$$
Furthermore, on $U_q$, using the fact that $\delta_{1/r}$ is a contact map and Lemma 3.3.10 in \cite{GClicentiate},
\begin{align*}
\nabla_{\mathbb{H}} \tilde{f} = \nabla_\mathbb{H} (  f \circ \delta_{1/r} ) = (\delta_{1/r})_*^T (\nabla_{\mathbb{H}}  f)_{\delta_{1/r}}  \neq 0.
\end{align*}
\end{proof}

\begin{prop}\label{letra}
Consider a left translation map $\tau_{\bar{p}} : \mathbb{H}^n \to \mathbb{H}^n$, $\bar{p} \in \mathbb{H}^n$, the anisotropic dilation $\delta_r : \mathbb{H}^n \to \mathbb{H}^n$, $r>0$ and a $\mathbb{H}$-regular $1$-codimensional surface $S \subseteq \mathbb{H}^n$. Then the $\mathbb{H}$-regular $1$-codimensional surfaces $\tau_{\bar{p}} S$  and $\delta_r S $ are $\mathbb{H}$-orientable (respectively) if and only if $S$ is $\mathbb{H}$-orientable.
\end{prop}

\begin{proof}
Remember that $\tau_{\bar{p}} S = \{ \bar{p}*p ; \ p \in S   \}$. From Proposition \ref{tau_prop}, one knows that for all $ q \in \tau_{\bar{p}} S$ there exists a point $ p \in S$ so that $q=\bar{p}*p$ and there exist a neighbourhood $U_p$ and a function $f: U_p \to \mathbb{R}$ so that $S \cap U_p = \{ f=0 \}$ and $\nabla_\mathbb{H} f \neq 0$ on $U_p$. 
Furthermore, $U_q = \tau_{\bar{p}} U_p = \bar{p} * U_p $ is a neighbourhood of $q=  \bar{p}*p$ and the function $\tilde{f}: = f \circ \tau_{\bar{p}}^{-1} : U_q \to \mathbb{R} $ is so that  $\tau_{\bar{p}} S \cap U_q = \{ \tilde{f}=0 \}$ and $\nabla_\mathbb{H} \tilde{f}  \neq 0$ on $U_q$.\\
Assume now that $S$ is $\mathbb{H}$-orientable, then there exists a global vector field
$$
n_{\mathbb{H}} = \sum_{j=1}^{n} \left (  n_{\mathbb{H},j} X_j + n_{\mathbb{H},n+j} Y_j  \right ),
$$
that, locally, takes the form of
$$
 \sum_{j=1}^{n}  \left ( 
\frac{ X_j f }{ \vert \nabla_\mathbb{H} f \vert }  X_j +  \frac{ Y_j f }{ \vert \nabla_\mathbb{H} f \vert }   Y_j \right ).
$$
Now we consider:
\begin{align*}
(\tau_{\bar{p}}^{-1})_*   n_\mathbb{H} = \sum_{j=1}^{n} \left (  
n_{\mathbb{H},j}  \circ \tau_{\bar{p}}^{-1} 
 {X_j}_{\tau_{\bar{p}}^{-1}} +
n_{\mathbb{H},n+j}  \circ \tau_{\bar{p}}^{-1} 
 {Y_j}_{\tau_{\bar{p}}^{-1}}  \right ),
\end{align*}
which, locally, becomes
\begin{align*}
(\tau_{\bar{p}}^{-1})_*   n_\mathbb{H} =
\sum_{j=1}^{n} \left (  
 \frac{ X_j f }{ \vert \nabla_\mathbb{H} f \vert }  \circ \tau_{\bar{p}}^{-1} 
 {X_j}_{\tau_{\bar{p}}^{-1}} +
\frac{ Y_j f }{ \vert \nabla_\mathbb{H} f \vert }  \circ \tau_{\bar{p}}^{-1} 
 {Y_j}_{\tau_{\bar{p}}^{-1}}  \right ).
\end{align*}
Note that this is still a global vector field and is defined on the whole $\tau_{\bar{p}} S$, therefore it gives an orientation to $\tau_{\bar{p}} S$. 
Since we can repeat the whole proof starting from $\tau_{\bar{p}} S$ to $S= \tau_{\bar{p}}^{-1} \tau_{\bar{p}} S$, this proves both directions.\\\\
For the dilation, remember that $\delta_r S = \{ \delta_r(p) ; \  p \in S   \}$. From Proposition \ref{delta_prop}, for all $ q \in \delta_r S$ there exists a point $ p \in S$ so that $q=\delta_r(p)$ and there exist a neighbourhood $U_p$ and a function $f: U_p \to \mathbb{R}$ so that $S \cap U_p = \{ f=0 \}$ and $\nabla_\mathbb{H} f \neq 0$ on $U_p$. 
In the same way, $U_q = \delta_r( U_p ) $ is a neighbourhood of $q=  \delta_r(p)$ and the function $\tilde{f}: = f \circ \delta_{1/r} : U_q \to \mathbb{R} $ is so that  $\delta_r S \cap U_q = \{ \tilde{f}=0 \}$ and $\nabla_\mathbb{H} \tilde{f}  \neq 0$ on $U_q$.\\
Assume now that $S$ is $\mathbb{H}$-orientable. Then there exists a global vector field
$$
n_\mathbb{H} = \sum_{j=1}^{n} \left (  n_{\mathbb{H},j}  X_j + n_{\mathbb{H},n+j}  Y_j  \right ),
$$
that, locally, is written as
$$
 \sum_{j=1}^{n}  \left ( 
\frac{ X_j f }{ \vert \nabla_\mathbb{H} f \vert }  X_j +  \frac{ Y_j f }{ \vert \nabla_\mathbb{H} f \vert }   Y_j \right ).
$$
Now we have
\begin{align*}
( \delta_{1/r} )_*   n_\mathbb{H} =  \sum_{j=1}^{n} \left (  
 n_{\mathbb{H},j}  \circ \delta_{1/r} 
 {X_j}_{\delta_{1/r}} +
n_{\mathbb{H},n+j}  \circ \delta_{1/r} 
 {Y_j}_{\delta_{1/r}}  \right ),
\end{align*}
which, locally,  becomes
\begin{align*}
( \delta_{1/r} )_*   n_\mathbb{H}  =
 \sum_{j=1}^{n} \left (  
 \frac{ X_j f }{ \vert \nabla_\mathbb{H} f \vert }  \circ \delta_{1/r} 
 {X_j}_{\delta_{1/r}} +
\frac{ Y_j f }{ \vert \nabla_\mathbb{H} f \vert }  \circ \delta_{1/r}
 {Y_j}_{\delta_{1/r}}  \right ).
\end{align*}
Note that this is still a global vector field and is defined on the whole $\delta_r S$, therefore it gives an orientation to $\delta_r S$. 
Since we can repeat the whole proof starting from $\delta_r S$ to $S=\delta_{1/r} \delta_r S$, this proves both directions.
\end{proof}


\subsubsection{Comparison}

If a surface is regular enough, we see that $\mathbb{H}$-orientability implies Euclidean-orientability. Lastly, this allows us to conclude that non-$\mathbb{H}$-orientable $\mathbb{H}$-regular surfaces exist, at least when $n=1$.\\

\noindent
Consider a $1$-codimensional $C^1$-Euclidean surface $S \subseteq \mathbb{H}^n$ with $C(S)\neq 0$. We say that $S$ is $C^2_\mathbb{H}$-regular if its horizontal normal vector field $n_\mathbb{H} \in C^1_\mathbb{H}$.

\begin{prop}\label{finalmente4}
Consider  a $1$-codimensional $C^1$-Euclidean surface $S$ in $\mathbb{H}^{n}$ with $C(S) = \varnothing$. Then the following holds: 
\begin{enumerate}
\item Suppose $S$ is Euclidean-orientable. Recall from condition \eqref{eq2} that $C^1$-Euclidean means that for all $ p \in S$ there exists $ U \ni  p$ 
 and $g : U \to \mathbb{R}$, $g \in C^1$, so that $S \cap U = \{ g=0 \}$ and $\nabla g \neq 0$  on $U$.
If, for any such $g$, no point of $S$ belongs to the set
$$
\left \{ 
\left ( 
  - \frac{2    ( \partial_{y_1}   g )_p  }{ ( \partial_t   g )_p  } 
, \dots, 
  - \frac{2    ( \partial_{y_n}   g )_p  }{ ( \partial_t   g )_p  } 
,
  \frac{2    ( \partial_{x_1}   g )_p  }{ ( \partial_t   g )_p  }
,\dots,
  \frac{2    ( \partial_{x_n}   g )_p  }{ ( \partial_t   g )_p  }
,t
 \right )
 , \text{ with }   ( \partial_t   g )_p \neq 0  \right \},
$$
then
\begin{equation}\label{.1}
 S \text{ is } \mathbb{H}\text{-orientable}.
\end{equation}
\item
If $S$ is  $C^2_\mathbb{H}$-regular,
\begin{equation}\label{.2}
S \text{ is } \mathbb{H}\text{-orientable implies }   S \text{ is Euclidean-orientable } .
\end{equation}
\end{enumerate}
\end{prop}

\noindent
The proof will follow at the end of this chapter. A question arises naturally about the extra conditions for the first implication in Proposition \ref{finalmente4}: what can we say about that set? Is it possible to do better?\\
Note also that, if we could simply assume that the functions $f$ and $g$, respectively of conditions \eqref{eq1} and \eqref{eq2}, would be the same, then we would not need the extra condition in the first implication and we would have that  Euclidean-orientability implies $\mathbb{H}$-orientability (see Lemma 4.3.23 in \cite{GClicentiate}).\\

\noindent
We have seen, by Proposition \ref{Mobius}, that $ \widetilde{\mathcal{M}}$ is a $\mathbb{H}$-regular surface and not Euclidean-orientable. Furthermore, $ \widetilde{\mathcal{M}}$ satisfies the hypotheses of Proposition \ref{finalmente4} and implication \eqref{.2} reversed says that $\widetilde{\mathcal{M}}$ is not a $\mathbb{H}$-orientable $\mathbb{H}$-regular surface. Then we can say:

\begin{cor}\label{cor:existence}
There exist $\mathbb{H}$-regular surfaces which are not $\mathbb{H}$-orientable in $\mathbb{H}^1$.
\end{cor}

\noindent
This opens the possibility to analysis of Heisenberg currents mod $2$ by studying surfaces that are, in the Heisenberg sense, regular but not orientable.

\begin{proof}[Proof of implication \eqref{.1} in Proposition \ref{finalmente4}]
We know there exists a global vector field 
$$ 
n_{E}=\sum_{i=1}^{n} \left ( n_{E,i} \partial_{x_i} +n_{E,n+i} \partial_{y_i} \right ) + n_{E,2n+1} \partial_t  \neq 0 
$$ 
that can be written locally on an open set $U\subseteq \mathbb{H}^n$ as 
$$
 n_{E}=\mu \sum_{i=1}^{n} \left (  \partial_{x_i} g  \partial_{x_i} +  \partial_{y_i} g  \partial_{y_i} \right ) + \mu \partial_t  g  \partial_t
$$
 so that $\nabla g  \neq 0$ and $g \in C^1(U,\mathbb{R})$.  
Define
$$
\begin{cases}
n_{\mathbb{H},i} :=n_{E,i} -\frac{1}{2}y_i \cdot n_{E,2n+1},\\
n_{\mathbb{H},n+i} := n_{E,n+i} +\frac{1}{2}x_i \cdot n_{E,2n+1},
\end{cases}
\quad  i=1,\dots,n.
$$
For each point $p$ there exists a neighbourhood $U$ where such $g$ is defined as above; locally in such sense, we get 
$$
\begin{cases}
n_{\mathbb{H},i} =\mu \partial_{x_i}   g -\frac{1}{2} y_i \mu \partial_t   g= \mu  X_i g,\\
n_{\mathbb{H},n+i} = \mu \partial_{y_i}   g +\frac{1}{2} x_i \mu \partial_t   g= \mu  Y_i g ,
\end{cases}
\quad  i=1,\dots,n,
$$
where $\mu$ is simply a normalising factor that, from now on, we ignore.\\
In order to verify the $\mathbb{H}$-orientability, we have to show that $\nabla_{\mathbb{H}} g \neq 0$ .  
Note here that $ C^1(U,\mathbb{R}) \subsetneq C_{\mathbb{H}}^1(U,\mathbb{R}), $ so $g$ is regular enough.\\ 
Consider first the case in which $( \partial_t   g )_p =0$. We still have that $\nabla_{p} g \neq 0$, so at least one of the derivatives $ ( \partial_{x_i}   g )_p , \ ( \partial_{y_i}   g )_p $ must be different from zero in $p$. But, when $( \partial_t   g )_p =0$, then  $(X_i g)_p= ( \partial_{x_i}   g )_p$ and $(Y_i g)_p= ( \partial_{y_i}   g )_p $, so
$$
\norm{ \nabla_{\mathbb{H},p} g }^2=(X_i g)_p^2 + (Y_i g)_p^2  \neq 0.
$$
\noindent
Second, consider the case when $( \partial_t   g )_p \neq 0  $. In this case:
$$
\norm{ \nabla_{\mathbb{H},p} g }^2 =\sum_{i=1}^{n} (X_i g)_p^2 + (Y_i g)_p^2 = \sum_{i=1}^{n}  \left  (  \partial_{x_i} g  -\frac{1}{2}y_{i,p}  \partial_t   g \right )_p^2 + \left (  \partial_{y_i} g + \frac{1}{2}x_{i,p}  \partial_t   g  \right )_p^2  \neq 0
$$
is equivalent to the fact that there exists $i \in \{1,\dots, n\}$ such that
$$
 y_{i,p}   \neq  \frac{2    ( \partial_{x_i}   g )_p  }{ ( \partial_t   g )_p  } \  \text{ or } \   x_{i,p}   \neq  - \frac{2    ( \partial_{y_i} g )_p  }{ ( \partial_t   g )_p  }.
$$
So the Heisenberg gradient of $g$ in $p$ is zero at the points\\
$$
\left ( 
  - \frac{2    ( \partial_{y_1}   g )_p  }{ ( \partial_t   g )_p  } 
, \dots, 
  - \frac{2    ( \partial_{y_n}   g )_p  }{ ( \partial_t   g )_p  } 
,
  \frac{2    ( \partial_{x_1}   g )_p  }{ ( \partial_t   g )_p  }
,\dots,
  \frac{2    ( \partial_{x_n}   g )_p  }{ ( \partial_t   g )_p  }
,t
 \right )
$$
and the first implication of the proposition is true.
\end{proof}


\begin{proof}[Proof of implication \eqref{.2} in Proposition \ref{finalmente4}]
In the second case \eqref{.2}, we know that there exists a global vector 
$$
 n_{\mathbb{H}}= \sum_{i=1}^{n} n_{\mathbb{H},i} X_i + n_{\mathbb{H},n+i} Y_i  \neq 0 
$$
 that can be written locally as 
$$
 n_{\mathbb{H}}= \sum_{i=1}^{n} \lambda X_i f   X_i + \lambda Y_i f  Y_i 
$$
 so that $\nabla_{\mathbb{H}} f  \neq 0 , \ f \in C_{\mathbb{H}}^1(U,\mathbb{R})$, with $U\subseteq \mathbb{H}^n$ open. As before, $\lambda$ is simply a normalising factor that, from now on, we ignore.\\
Note that $n_{\mathbb{H}} \in C_{\mathbb{H}}^1(U,\mathbb{R})$, which is the same as asking $S$ to be $C_\mathbb{H}^2$-regular. Then define
$$
\begin{cases}
n_{E,2n+1}:=\frac{1}{n} \sum_{j=1}^{n} \left ( X_j n_{\mathbb{H},n+j} - Y_j n_{\mathbb{H},j} \right ), \\
n_{E,i} := n_{\mathbb{H},i} +\frac{1}{2}y_i \cdot n_{E,2n+1}, \\
n_{E,n+i} := n_{\mathbb{H},n+i} -\frac{1}{2}x_i \cdot n_{E,2n+1} ,
\end{cases}
\quad  i=1,\dots,n.
$$
For each point $p$ there exists a neighbourhood $U$ where such $f$ is defined. Locally in such sense, we can write the above as:
$$
n_{E,2n+1}
= \frac{1}{n}  \sum_{j=1}^{n} \left (   X_j Y_j f - Y_j   X_j f \right ) 
=\frac{1}{n} n  Tf=  \partial_t f.
$$
So now we have that
$$
\begin{cases}
n_{E,2n+1}=   \partial_t f,\\
n_{E,i} =  \partial_{x_i}  f -\frac{1}{2} {y_i}  \partial_t  f  +\frac{1}{2} {y_i}  \partial_t  f=  \partial_{x_i}  f ,\\
n_{E,n+i} =  \partial_{y_i}  f +\frac{1}{2} {x_i}  \partial_t  f  -\frac{1}{2} {x_i} \partial_t  f =  \partial_{y_i}  f ,
\end{cases}
\quad  i=1,\dots,n.
$$
In order to verify the Euclidean-orientability, we have to show that $\nabla f \neq 0$ . \\
Note that $f \in C_{\mathbb{H}}^1(U,\mathbb{R})$ and, a priori, we do not know whether $f \in C^1(U,\mathbb{R})$. However, asking $n_{\mathbb{H}} \in C_{\mathbb{H}}^1(U,\mathbb{R})$ allows us to write $\partial_{x_i}, \partial_{y_i} $ and $ \partial_t$ using only $X_i, Y_i, n_{\mathbb{H},i} $  and  $ n_{\mathbb{H},n+i}$, which guarantees that  $\partial_{x_i} f, \partial_{y_i} f$ and $ \partial_t f$ are well defined.\\
Now, $\nabla f \neq 0$ if and only if
\begin{align*}
\sum_{i=1}^{n} \left ( (\partial_{x_i} f)^2 + (\partial_{y_i} f)^2 \right ) + (\partial_t f)^2 \neq 0 ,
\end{align*}
which is the same as
\begin{align*}
& \sum_{i=1}^{n} \left [ \left  ( X_i f + \frac{1}{2}y_i  T f \right  )^2 + \left ( Y_i f - \frac{1}{2}x_i T f \right   )^2 + \left ( T f \right )^2 \right ] \neq 0 .
\end{align*}
In the case  $Tf \neq 0$, we have that $\nabla f \neq 0$ immediately. In the case $Tf=0$, instead, we have that $\nabla f \neq 0$ if and only if
$$
 \sum_{i=1}^{n} \left [ \left ( X_i f \right )^2+ \left ( Y_i f  \right   )^2 \right ] \neq 0,
$$
which is true because $\nabla_{\mathbb{H}} f \neq 0$. This completes the cases and shows that there actually is a global vector field $n_E$ that is continuous (by hypotheses) and never zero. So the second implication of the proposition is true.
\end{proof}


\bibliography{Bibliography_United_Thesis_190520} 
\bibliographystyle{abbrv}

\end{document}